\documentclass[12pt]{article}
\usepackage{geometry}                
\usepackage{graphicx,color}
\usepackage{amssymb,amsmath,amsthm,mathrsfs}
\usepackage[all,cmtip]{xy}
\usepackage{epstopdf, comment}

\usepackage[pdftex,bookmarks,pdfnewwindow,plainpages=false,unicode,pdfencoding=auto]{hyperref}

 \hypersetup{
pdfauthor={Oleg Ivrii},
pdftitle={Prescribing inner parts of derivatives of inner functions},
pdfsubject={Function theory in the unit disk},
bookmarksdepth={4}
}

\numberwithin{equation}{section}

\newtheorem{theorem}{Theorem}[section]

\newtheorem{lemma}[theorem]{Lemma}
\newtheorem{corollary}[theorem]{Corollary}

\theoremstyle{remark}
\newtheorem*{remark}{Remark}

\theoremstyle{definition}
\newtheorem*{definition}{Definition}

\DeclareMathOperator{\inn}{Inn}
\DeclareMathOperator{\out}{Out}

\DeclareMathOperator{\aut}{Aut}

\DeclareMathOperator{\gap}{gap}

\DeclareMathOperator{\escape}{\escape}

\DeclareMathOperator{\vis}{con}
\DeclareMathOperator{\inv}{inv}

\title{Prescribing inner parts \\ of derivatives of inner functions}
\author{Oleg Ivrii}
\date{August 8, 2017}

\begin{document}

\maketitle

\begin{abstract}
Let $\mathscr J$ be the set of inner functions whose derivative lies in Nevanlinna class. In this note, we show that the  natural map $$F \to \inn(F') \quad : \quad \mathscr J/\aut(\mathbb{D}) \to \inn/S^1$$
is injective but not surjective. More precisely, we show that that the image consists of all inner functions of the form $BS_\mu$ where $B$ is a Blaschke product and $S_\mu$ is the singular factor associated to a measure $\mu$ whose support is contained in a countable union of Beurling-Carleson sets. Our proof is based on extending the work of D.~Kraus and O.~Roth on maximal Blaschke products to allow for singular factors. This  answers a question raised by K.~Dyakonov.
\end{abstract}

\section{Introduction}

\enlargethispage{2pt}

Consider the following curious differentiation procedure: to a Blaschke product
$$
F(z) = \prod_{i=1}^d \frac{z-a_i}{1-\overline{a_i}z}
$$
of degree $d \ge 1$, one can  associate a Blaschke product $B$ of degree $d-1$ whose zeros are located at the critical points of $F$ (that is, at the zeros of $F'$). It is a classical result of M.~Heins \cite[Section 29]{heins} that this correspondence is a bijection, provided one considers $F$ modulo post-composition with M\"obius transformations (as not to change its critical set)
and $B$ up to rotations (which preserve the zero set).

In this paper, we discuss an infinite degree analogue of this problem posed by K.~Dyakonov in \cite{dyakonov-mobius, dyakonov-inner}. We need some definitions.
An {\em inner function} is a holomorphic self-map of the unit disk such that for almost every $\theta \in [0, 2\pi)$, the radial limit $\lim_{r \to 1} F(re^{i\theta})$ exists and has absolute value 1. Let $\inn$ denote the space of all inner functions. We will also be concerned with the subclass $\mathscr J$ of inner functions whose derivative lies in Nevanlinna class, i.e.~which satisfy
\begin{equation}
\label{eq:finite-entropy}
\lim_{r \to 1} \frac{1}{2\pi} \int_{0}^{2\pi} \log^+ |F'(re^{i\theta})| d\theta < \infty.
\end{equation}
Together with Jensen's formula,  (\ref{eq:finite-entropy}) implies that the set of critical points $\{c_i\}$ of $F$ satisfies the {\em Blaschke condition} $\sum (1-|c_i|) < \infty$, and is therefore the zero set of some Blaschke product, which could be either finite or infinite.

According to the work of Ahern and Clark, if $F'$ is a Nevanlinna class function, then it admits an ``inner-outer'' decomposition $F' = \inn F' \cdot \out F'$, see Lemma \ref{ac-lemma} below. The mapping $F \to \inn F'$ from $\mathscr J$ to $\inn$ generalizes the construction outlined for finite Blaschke products above, however,  in addition to recording the critical set of $F$, $\inn F'$ may also contain a non-trivial singular factor. This feature allows us to distinguish different Blaschke products with the same critical set. In this paper, we prove the following theorem:

\begin{theorem}
\label{main-thm}
Let $\mathscr J$ be the set of inner functions whose derivative lies in Nevanlinna class. The natural map $$F \to \inn(F') \quad : \quad \mathscr J/\aut(\mathbb{D}) \to \inn/S^1$$
is injective. The image consists of all inner functions of the form $BS_\mu$ where $B$ is a Blaschke product and $S_\mu$ is the singular factor associated to a measure $\mu$ whose support is contained in a countable union of Beurling-Carleson sets.
\end{theorem}

In \cite{dyakonov-mobius}, K.~Dyakonov showed that $\inn F'$ is trivial if and only if $F$ is a M\"obius transformation.
After reading Dyakonov's work, the author realized that a theorem of D.~Kraus can be reformulated as ``$F \to \inn F'$ is a bijection from Maximal Blaschke Products in $\mathscr J$ to the space of all Blaschke Products.'' 
The main focus of this paper will be to understand the role of  singular factors. 

\subsection{Strategy}

We now state several propositions which will be used to show Theorem \ref{main-thm}. These will be proved in Sections \ref{sec:understanding-image} and \ref{sec:roberts} after we develop the necessary tools.

\begin{lemma}[Decomposition rule] An inner function $B_C S_\mu$ lies in the image of $F \to \inn F'$ if and only if its singular part $S_\mu$ does.
\end{lemma}

Therefore, to describe the image of our mapping, it suffices to determine which singular inner functions $S_\mu$ can be represented as $S_\mu = \inn F'_\mu$ with $F_\mu \in \mathscr J$. If such an $F_\mu$ can be found (which is necessarily unique), we say that the measure $\mu$ is {\em constructible}\/.

\begin{lemma}[Product rule]
\label{product-law2}
Suppose measures $\mu_j$, $j=1, 2, \dots$ are constructible.
If their sum $\mu = \sum_{j=1}^\infty \mu_j$ is finite, then $\mu$ is also constructible.
\end{lemma}

\begin{lemma}[Division rule]
\label{divisors3}
If a measure $\mu$ is constructible, then any $\nu \le \mu$ is also constructible.
\end{lemma}

Recall that a {\em Beurling-Carleson set} is a closed subset of the unit circle of zero Lebesgue measure whose complement is a union of arcs $\bigcup_k I_k$ with $\sum |I_k| \log \frac{1}{|I_k|} < \infty$. 
To obtain a large supply of constructible measures, we use the following result of
 Cullen  \cite{cullen}:

\begin{lemma}
\label{cullen-thm}
Suppose the support of $\mu$ is contained in a Beurling-Carleson set.
Then $S'_\mu \in \mathcal N$.
\end{lemma}

Since $S_\mu$ divides $S'_\mu$, the division rule implies that any measure $\mu$ supported on a Beurling-Carleson set is constructible. By the product rule, any measure  supported on a countable union of
 Beurling-Carleson sets is also constructible. Theorem \ref{main-thm} states that any constructible measure is of this form.
Moreover, Theorem \ref{main-thm} implies that Cullen's theorem is essentially sharp:

\begin{corollary}
\label{cullen-sharp}
Suppose $\mu$ is a measure on the unit circle with $S'_\mu \in \mathcal N$. Then, the support of $\mu$ is contained in a countable union of Beurling-Carleson sets.
\end{corollary}

On the other side of the spectrum, we have invisible measures.
 We say that a finite positive singular measure $\mu$ is {\em invisible} if for any measure
 $0 <  \nu \le \mu$, there does not exist a function $F_\nu \in \mathscr J$ with $\inn F_\nu' = S_\nu$.
In Section \ref{sec:understanding-image}, we will show that any singular measure on the unit circle $\mu$ can be uniquely decomposed into a constructible part and an invisible part: $\mu = \mu_{\vis} + \mu_{\inv}$. To complete the proof of Theorem \ref{main-thm}, we give a criterion for a measure to be invisible:

\begin{theorem}
\label{not-charging-thm}
Suppose $\mu$ is a measure on the unit circle which does not charge Beurling-Carleson sets. Then, it is invisible.
\end{theorem}

  The reader may notice that the notion of an invisible measure coincides with the description of cyclic inner functions in Bergman spaces given independently by Korenblum \cite{korenblum} and Roberts \cite{roberts}. To prove Theorem \ref{not-charging-thm}, we will first show that any measure $\mu$ with modulus of continuity $\omega(t) \le Ct\log(1/t)$ is invisible.
To obtain the full result, we use an iterative scheme based on a clever decomposition of a measure that does not charge Beurling-Carleson sets into ``$t\log1/t$''-pieces from \cite{roberts}. 

We conclude the introduction by spending a moment to check that the map in Theorem \ref{main-thm} is well-defined:

\begin{lemma}
\label{frostman-derivative}
If $F \in \mathscr J$ is an inner function, then for any M\"obius transformation $T \in \aut \mathbb{D}$, the Frostman shift $T \circ F \in \mathscr J$ and $\inn (T \circ F)' = \inn F'$.
\end{lemma}

\begin{proof}
From the chain rule, we have $(T \circ F)'(z) =  T'(F(z)) \cdot F'(z)$. Since $\log |T'|$ is bounded, $T \circ F \in \mathscr J$. The equality also tells us that the inner part $\inn (T \circ F)'$ is divisible by $\inn F'$. Using $T^{-1}$ in place of $T$, we see that
$\inn F'$ is divisible by $\inn (T \circ F)'$. Hence, $\inn (T \circ F)' = \inn F'$ agree (up to a unimodular constant).
\end{proof}

\subsection{Notation}

Let $m$ denote the Lebesgue measure on ${\mathbb{S}}^1$, normalized to have unit mass and
$\lambda_{\mathbb{D}} = \frac{|dz|}{1-|z|^2}$ be the Poincar\'e metric on the unit disk. For a holomorphic mapping $F: \mathbb{D} \to \mathbb{D}$, we denote the associated conformal metric of  constant curvature $-4$ by
$$
\lambda_F := \frac{|F'|}{1-|F|^2}.
$$
Given a Blaschke sequence $C$ in the unit disk, let $B_C$ be the Blaschke product with zero set $C$ and $F_C$ denote the maximal Blaschke product with critical set $C$. In order for $B_C$ and $F_C$ to be uniquely defined, we use the normalizations $B_C(1) = 1$ and $F_C(0) = 0$, $F_C'(0) > 0$ (or $F_C^{(n+1)}(0) > 0$ if $C$ contains 0 with multiplicity $n$).
 For a singular measure $\mu$ on the unit circle, we let $S_\mu$ be the associated singular inner function.

\section{Background on conformal metrics}

 Given an at most countable set $C$ in the unit disk (counted with multiplicity), the machinery of Kraus and Roth \cite{kraus}--\cite{maximal-blaschke} seeks to construct a Blaschke product with critical set $C$. If such a Blaschke product does not exist, then the machinery
 does not produce anything. If there are Blaschke products with critical set $C$, the machinery produces the optimal or {\em maximal} Blaschke product $F_C$.

 Instead of constructing $F_C$ directly, Kraus and Roth construct the conformal metric $F_C^* \lambda_{\mathbb{D}}$ -- the pullback of the Poincar\'e metric on the disk. We give a brief overview of their construction. Following Heins, an {\em SK-metric} $\lambda(z) |dz|$ is a conformal pseudometric on a domain $U$ whose density $\lambda: U \to [0,\infty)$ is a continuous function with  curvature
$$
k_\lambda = - \frac{\Delta \log \lambda}{\lambda^2} \le -4,
$$
in the sense of distributions. 
According to \cite[Section 13]{heins} or \cite[Definition 4.11]{conf-metrics}, a collection $\Phi$ of SK-metrics is a {\em Perron family} if it is closed under {\em modifications} and {\em taking maxima}\/. The first condition means that given a round disk $D \subset U$ and a metric $\lambda \in \Phi$, the (unique) SK-metric $M_D \lambda$ which agrees with $\lambda$ on $U \setminus D$ and has  curvature $-4$ in $D$ lies in $\Phi$; while the second condition says that for any $\lambda_1, \lambda_2 \in \Phi$, their pointwise maximum
 $\max(\lambda_1, \lambda_2)$ is also in $\Phi$.
Heins proved that if a Perron family is non-empty, then the supremum of all metrics in $\Phi$ is a regular conformal metric of curvature $-4$, where regular means ``$C^2$ on the complement of $\{ z \in U: \lambda(z) = 0 \}$.''

We also recall a complementary theorem due to Liouville \cite[Theorem C]{KR-survey} which says that if a conformal
metric $\lambda(z) |dz|$ has constant curvature $-4$ and all  its zeros $c_i$ have {\em integral} multiplicities, that is if
$$
\lim_{z\to c_i} \frac{\lambda(z)}{|z-c_i|^{m_i}} = L_i, \qquad \text{for some} \quad 1 \le m_i \in \mathbb{Z}, \quad 0 < L_i < \infty,
$$
 then $\lambda(z)|dz|$  is necessarily of the form $\lambda_F = F^* \lambda_{\mathbb{D}}$ for some holomorphic function $F: U \to \mathbb{D}$. Furthermore, the function $F$ is unique up to post-composition with a M\"obius transformation.

For a set $C$ in the unit disk, let $\Phi_C$ be the collection of all conformal metrics vanishing on $C$. It clearly verifies the two axioms of being a Perron family on the domain $\mathbb{D} \setminus C$.
Provided $\Phi_C$ is non-empty, one obtains a metric of constant curvature $-4$ and a holomorphic function $F_C$ which vanishes on $C$ to the correct order. Leveraging the maximality of the metric $\lambda_{F_C}$, Kraus \cite{kraus} proved that the outer and singular inner factors of $F_C$ are trivial. In other words, $F_C$ is a Blaschke product.

In the case when the critical set $C$ is a Blaschke sequence, Kraus made the fundamental observation that $|B_C|  \lambda_{\mathbb{D}}$ is an SK-metric which guarantees  that the Perron family $\Phi_C$ is non-empty. (More generally, given a holomorphic function $H$ with $\|H\|_\infty \le 1$ and a metric $\lambda$ of curvature $-4$,  $|H| \cdot \lambda$ is an SK-metric.)

 Further exploiting the lower bound $\lambda_{F_C} \ge |B_C|  \lambda_{\mathbb{D}}$, Kraus obtained the following remarkable result \cite[Theorem 4.4]{kraus}:

\begin{theorem}[Kraus]
\label{kr-theorem}
 Suppose $C$ is a Blaschke sequence in the disk and $\lambda$ is a metric of constant curvature $-4$ which vanishes precisely at $C$ with the correct multiplicity. Then $\lambda = \lambda_{F_C}$ if and only if
\begin{equation}
\label{eq:kr-theorem}
\lim_{r \to 1} \int_{|z|=r} \log \frac{\lambda}{\lambda_{\mathbb{D}}} \, d\theta = 0.
\end{equation}
\end{theorem}

In Section \ref{sec:gap}, we will use ideas of Ahern and Clark to show that the above theorem can be alternatively formulated as:

\begin{corollary}
\label{kr-corollary}
Suppose $C$ is a Blaschke sequence in the disk.
An infinite Blaschke product $F \in \mathscr J$ is the maximal Blaschke product associated to $C$ if and only if the singular factor of $\inn F'$ is trivial, i.e.~if $\inn F' = B_C$.
\end{corollary}

In order to generalize the arguments of Kraus and Roth to allow for singular factors, we will need:
\begin{lemma}[Fundamental Lemma]
\label{fundamental-lemma}
For any inner function $F \in \mathscr J$,
\begin{equation}
\label{eq:fundamental-lemma}
\lambda_F \ge |\inn F'| \lambda_{\mathbb{D}}.
\end{equation}
In fact, $\lambda_F$ is the smallest metric of constant curvature $-4$ with this property.
\end{lemma}

Note that the minimality of the metric $\lambda_F$ implies that the map $F \to \inn F'$ from Theorem \ref{main-thm} is injective.
As explained above, the inequality (\ref{eq:fundamental-lemma}) holds for maximal Blaschke products. In Section \ref{sec:stable}, we will deduce the general case by considering finite approximations.

Using the factorization $F' = \inn F' \cdot \out F'$, one can rewrite (\ref{eq:fundamental-lemma}) as
\begin{equation}
\label{eq:dyakonov-reverse}
\frac{1-|F(z)|}{1-|z|} \le |\out F'|,
\end{equation}
which was first proved by Dyakonov in  \cite[Theorem 2.1]{dyakonov-coinvariant} using Julia's lemma.
The reader may also consult \cite[Corollary 2.1]{dyakonov-reverse} for additional remarks.
In view of the above discussion, the fundamental lemma may be viewed as a refinement of Dyakonov's theorem.

\subsection{Wedge of two metrics}
\label{sec:wedge}
Given two inner functions $F, G \in\mathscr J$,  consider the family
$\Phi_{F, G}$ of SK-metrics that are pointwise less than $\min(\lambda_F, \lambda_G)$.
This family is not empty: the metric $|\inn F'| \cdot |\inn G'| \cdot \lambda_{\mathbb{D}}$ is in it, as Lemma \ref{fundamental-lemma} shows.
Taking the supremum of conformal metrics in $\Phi_{F, G}$, we get a regular conformal metric $\lambda_{F \wedge G}$ of constant curvature $-4$. Therefore, it is the pullback of $\lambda_{\mathbb{D}}$ by a holomorphic function which we denote $H = F \wedge G$.

To see that the outer part of $H$ is trivial, i.e.~that $H$ is inner, we can use the clever argument of Kraus \cite[Proof of Theorem 1.2]{kraus}. The relevant equation here is:

\begin{equation}
\Biggl [ \frac{|H'|}{|H| \log \frac{1}{|H|}} \cdot |\inn H| \Biggr ] \cdot |\inn F'| \cdot |\inn G'| \, \le \, \frac{|H'|}{1-|H|^2},
\end{equation}
where the expression in the square brackets is bounded above by $\lambda_{\mathbb{D}}$ since it is an SK-metric (see \cite[Lemma 2.17]{kraus}).
One finds a contradiction by examining the behaviour of both sides as $z \to e^{i\theta}$ radially to a  point on the unit circle at which $|\out H(e^{i\theta})| < 1$ and
$|\inn H(e^{i\theta})| = |\inn F'(e^{i\theta})| = |\inn G'(e^{i\theta})| = 1$.


\subsection{Hull of a conformal metric}
\label{sec:hull}

For an SK-metric $\kappa$, let $\Psi_\kappa$ be the collection of all metrics of constant  curvature $-4$ which are greater than $\kappa$ and $\Phi_\kappa$ be the collection of all SK-metrics that are less than all metrics in $\Psi_\kappa$. Since $\Phi_\kappa$ is a Perron family, its supremum is a metric $\hat \kappa$ of curvature $-4$. We call $\hat \kappa$ the {\em hull} of $\kappa$. From the definition, it is clear that $\hat \kappa$ is the smallest metric of curvature $-4$ which exceeds $\kappa$.
In this terminology, Lemma \ref{fundamental-lemma} says that $\lambda_F$ is the hull of $|\inn F'| \lambda_{\mathbb{D}}$.

\section{Gap of a Nevanlinna function}
\label{sec:gap}

By definition, the Nevanlinna class $\mathcal N$ consists of holomorphic functions on the unit disk for which
\begin{equation}
\label{eq:nevanlinna-condition}
\sup_{0 < r < 1} \frac{1}{2\pi} \int_{|z|=r} \log^+|f(z)| d\theta < \infty,
\end{equation}
see for instance  \cite{duren}.
It is well known that (unless $f$ is identically zero) this condition is equivalent to the boundedness of
$$
\sup_{0 < r < 1} \frac{1}{2\pi} \int_{|z|=r} \bigl | \log|f(z)| \bigr | d\theta.
$$
Since $\log |f(z)|$ is a subharmonic function, $\lim_{r \to 1}  \frac{1}{2\pi} \int_{|z|=r} \log|f(z)| d\theta$ exists and is finite.
However, unlike the Hardy norms, it need not be the case that
$$
\lim_{r \to 1}  \frac{1}{2\pi} \int_{|z|=r} \log|f(z)| d\theta \, = \,   \frac{1}{2\pi} \int_{|z|=1} \log|f(z)| d\theta,
$$
where in the integral in the right hand side, we consider the radial boundary values of $f$ which are known to exist a.e.
To understand the cause of the discrepancy, we consider  the canonical decomposition of $f = B(S/S_1)O$ into a Blaschke product, a quotient of singular inner functions and an outer function:
\allowdisplaybreaks
\begin{align*}
B &= \prod_i -\frac{\overline{a_i}}{|a_i|} \cdot \frac{z-a_i}{1-\overline{a_i}z}, \\
S/S_1 & = \exp \biggl ( - \int_{\mathbb{S}^1} \frac{\zeta+z}{\zeta-z} d\sigma_\zeta \biggr ), \qquad \sigma \perp m, \\
O & = \exp \biggl ( \int_{\mathbb{S}^1} \frac{\zeta+z}{\zeta-z}  \log |f(\zeta)| dm_\zeta \biggr ).
\end{align*}
Given an interval $I$ on the unit circle, let $rI$ denote its radial projection onto the circle
$S_r = \{z: |z|=r\}$. Fubini's theorem and the dominated converge theorem show:
\begin{lemma}
\label{nevanlinna-gap}
$$
 \gap(f) \, := \, \frac{1}{2\pi} \int_{|z|=1}  \log|f(z)| d\theta -
\lim_{r \to 1} \biggl \{ \frac{1}{2\pi} \int_{|z|=r}  \log|f(z)| d\theta \biggr \} \, = \, \sigma(\mathbb{S}^1).
$$
More generally, if $I$ is an interval on the unit circle,
$$
 \gap_I(f) \, := \, \frac{1}{2\pi} \int_{I}  \log|f(z)| d\theta -
\lim_{r \to1} \biggl\{ \frac{1}{2\pi} \int_{rI}  \log|f(z)| d\theta \biggr\} \, = \, \sigma(I),
$$
provided the endpoints of $I$ do not charge $\sigma$.
\end{lemma}

\subsection{Applications to inner functions}
\label{sec:inner-derivative}

We now apply Lemma \ref{nevanlinna-gap} to study inner functions with derivative in Nevanlinna class.
We first give a slightly different perspective on a classical theorem due to Ahern and Clark:

\begin{lemma}[Ahern-Clark] 
\label{ac-lemma}
For an inner function $F \in \mathscr J$, its derivative admits a $BSO$ decomposition. In other words, the singular measure $\sigma(F') \ge 0$.
\end{lemma}

\begin{proof}
By Lemma \ref{frostman-derivative}, it suffices to consider the case when $F(0) = 0$. Then $|F'(x)| \ge 1$ on the unit circle, e.g.~see \cite[Theorem 4.15]{mashreghi}. In view of the fundamental inequality $$|F'(rx)| \le 4|F'(x)|, \qquad x \in \mathbb{S}^1, \quad 0 < r <1,$$ of Ahern and Clark \cite{ahern-clark}, the dominated convergence theorem  shows
$$
\int_{I} \log^+|F'(z)| dm - \lim_{r \to 1} \int_{rI} \log^+|F'(z)| dm =  0,
$$
for any interval $I \subset \mathbb{S}^1$. 
However, by Fatou's lemma, the negative part of the logarithm can only dissipate and therefore
$$
\gap_I(F') = \int_{I} \log|F'(z)| dm - \lim_{r \to 1} \int_{rI} \log|F'(z)| dm \ge  0.
$$
This completes the proof.
\end{proof}

The following lemma relates the notions $\gap(F')$ and $\lambda_F$\,:

\begin{lemma}
\label{derivative-gap}
Let $I \subset \mathbb{S}^1$ be an interval. 
If $F \in \mathscr J$ then 
\begin{equation}
\label{eq:derivative-gap}
\frac{1}{2\pi} \int_{I} \log|F'(z)| d\theta = \lim_{r \to 1} \frac{1}{2\pi} \int_{rI} \log \frac{1-|F(z)|^2}{1-|z|^2} d\theta.
\end{equation}
\end{lemma}

\begin{proof}
From the contraction of the hyperbolic distance $d_\mathbb{D}(F(0),F(z)) \le d_\mathbb{D}(0,z)$, it follows that  the quotient $\frac{1-|F(z)|^2}{1-|z|^2} \ge c_{F(0)}$ is bounded below by a positive constant. 
By the Schwarz lemma,
$$
 \frac{1}{2\pi} \int_{rI} \max \bigl (\log|F'(z)|, \, \log c_{F(0)} \bigr ) d\theta
 \le  \frac{1}{2\pi} \int_{rI} \log \frac{1-|F(z)|^2}{1-|z|^2} d\theta.
$$
Applying the dominated convergence theorem like in the proof of Lemma \ref{ac-lemma} gives the $\le$ inequality in (\ref{eq:derivative-gap}).
For the $\ge$ direction, we average Dyakonov's inequality (\ref{eq:dyakonov-reverse})
 over $z \in rI$\,:
$$
\frac{1}{2\pi} \int_{rI} \log |\out F'(z)| d\theta
\ge \frac{1}{2\pi} \int_{rI} \log \frac{1-|F(z)|^2}{1-|z|^2} d\theta.
$$
The lemma follows after taking $r \to 1$ since $\log |\out F'(z)|$ is the harmonic extension of $\log|F'(z)|$ considered as a function on the unit circle.
\end{proof}

The reader may compare the above lemma with \cite[Theorem 3]{AAN}.

\subsection{Applications to conformal metrics}
\label{sec:ac-theory}

\begin{lemma}
\label{ac-theory1}
Suppose $F \in \mathscr J$ is an inner function for which
\begin{equation}
\lambda_F \ge |B_C S_\mu| \cdot \lambda_{\mathbb{D}}.
\end{equation}
Then, the singular measure
$\label{eq:inj2}
\sigma(F') \le \mu.
$
\end{lemma}

\begin{proof}
Let $I \subset \mathbb{S}^1$ be an interval. From the definition of $\lambda_F$,
\begin{equation*}
 \int_{rI} \log \frac{\lambda_F}{|B_C S_\mu| \lambda_{\mathbb{D}}}  \, dm =
  \int_{rI} \log \biggl ( \frac{|F '| (1-|z|^2)}{1 - |F |^2} \biggr )  \, dm
  -  \int_{rI} \log |B_C S_\mu| dm.
\end{equation*}
By Lemmas \ref{nevanlinna-gap}
 and \ref{derivative-gap}, as $r \to 1$, this tends to
 $$0 \le - \sigma(F')(I) + \sigma(S_\mu)(I),
 $$
 at least if $I$ is generic (there are extra terms if the endpoints of $I$ charge any of these singular measures).
\end{proof}

\begin{remark}
The same conclusion holds under the seemingly weaker assumption $\lambda_F \ge |B_C S_\mu O_h|$ where 
$$
O_h  = \exp \biggl ( \int_{\mathbb{S}^1} \frac{\zeta+z}{\zeta-z} \, h(\zeta)dm_\zeta \biggr ), \qquad h: \mathbb{S}^1 \to \mathbb{R},
$$
 is an arbitrary outer function: the above computation results in $\sigma(F') \le \mu - h \, dm$. Since $\sigma(F') \perp h \,dm$ are mutually singular, we have $\sigma(F') \le \mu$ and $h \le 0$.
\end{remark}

Similar considerations show:
\begin{lemma}
\label{ac-theory2}
If $F,G \in \mathscr J$ and the interval $I \subset \mathbb{S}^1$  is generic for both $\sigma(F')$ and $\sigma(G')$, then
\begin{equation}
\lim_{r\to 1} \int_{|z|=r} \log (\lambda_F / \lambda_{G}) dm \, = \, - \sigma(F')(I) + \sigma(G')(I).
\end{equation}
In particular, if $\lambda_F \ge \lambda_G$ then $\sigma(F') \le \sigma(G')$.
\end{lemma}

Combining the above lemma with Theorem \ref{kr-theorem} gives Corollary \ref{kr-corollary}.

\begin{lemma}
\label{technical-lemma}
If $\lambda_G$ is a metric of curvature $-4$  such that $\lambda_G \ge |H| \lambda_{\mathbb{D}}$ for some 
 bounded holomorphic function $H \not\equiv 0$, then $G \in \mathscr J$.
\end{lemma}

\begin{proof}
Since $H$ is a bounded holomorphic function, $\gamma_1 = \lim_{r \to 1} \int_{r\mathbb{S}^1}\log |H| dm$ is  finite.
The condition  $\lambda_G \ge |H| \lambda_{\mathbb{D}}$  implies that the zeros of $G'$ form a Blaschke sequence, which in turn implies that the integral $\gamma_2 = \lim_{r \to 1} \int_{r\mathbb{S}^1}\log |G'| dm$ is also finite.
An inspection of  the inequality
$$
0 \, \le \, \liminf_{r \to 1} \int_{r\mathbb{S}^1}\log \frac{\lambda_G}{|H|\lambda_\mathbb{D}} dm
\, \le \, -\gamma_1 + \gamma_2 - \limsup_{r \to 1} \int_{r\mathbb{S}^1} \log^+ |G'| dm
$$
 then shows that $G'$ satisfies the Nevanlinna condition (\ref{eq:nevanlinna-condition}).
 It remains to prove that the outer part of $G$ is trivial, so that $G$ is an inner function. 
 If this were not the case, then for a positive measure set of directions
$\theta \in [0,2\pi)$, $\limsup_{r \to 1} \lambda_G (re^{i\theta})$ would be finite.
However, this contradicts the assumption $\lambda_G \ge  |H| \lambda_{\mathbb{D}}$, since by the Lusin-Privalov theorem,  the radial limit of $H(re^{i\theta})$ is non-zero almost everywhere.
\end{proof}
\subsection{Injectivity and minimality}
\label{sec:inject}

We now show the injectivity part of Theorem \ref{main-thm}.
If there were two functions $F, G \in \mathscr J$  with $\inn F' = \inn G' =  B_C S_\mu$, then
\begin{equation}
\lambda_F \, \ge \, \lambda_{F \wedge G} \, \ge \,  |B_C S_\mu| \cdot \lambda_{\mathbb{D}}.
\end{equation}
Lemmas \ref{ac-theory1} and \ref{ac-theory2} imply that $(F \wedge G)'$ has the same inner part as $F'$. From the definition of curvature, $\Delta \log (\lambda_F / \lambda_{F \wedge G}) = 4(\lambda_F^2 - \lambda_{F \wedge G}^2)$. Hence
$\log (\lambda_F / \lambda_{F \wedge G})$ is subharmonic and non-negative, yet
\begin{equation}
\lim_{r\to 1} \int_{|z|=r} \log (\lambda_F / \lambda_{F \wedge  G}) dm \, \to \, 0,
\end{equation}
 which forces
$ \log (\lambda_F / \lambda_{F \wedge  G}) = 0$. We deduce that $\lambda_F = \lambda_{F  \wedge G} = \lambda_G$ and therefore $F = G$ up to post-composition with a M\"obius transformation by \cite[Theorem 5.1]{conf-metrics}.

Minimality is similar.
Given an inner function $F \in \mathscr J$, we now show that  $\lambda_F$ is the smallest metric of constant curvature $-4$ that exceeds $|\inn F'| \lambda_{\mathbb{D}}$.
Following Section \ref{sec:hull}, we consider the hull  $\lambda$ of the metric $|\inn F'| \lambda_{\mathbb{D}}$. 
The inequalities
\begin{equation}
\label{eq:min-ineq}
\lambda_F \, \ge \,  \lambda \, \ge \,  |\inn F'| \lambda_{\mathbb{D}}
\end{equation}
reveal that $\lambda$ has exactly the same zero set as $\lambda_F$ (counted with multiplicity). In particular, all zeros of $\lambda$ have integral multiplicities. 
Proceeding like in the proof of injectivity, we obtain  
$\lim_{r\to 1} \int_{|z|=r} \log (\lambda_F / \lambda) dm 
\,  \to \,  0$ and $\lambda = \lambda_F$ as desired.

\section{Stable approximations}
\label{sec:stable}

In this section, we study approximations of inner functions by finite and maximal Blaschke products. We are particularly interested in {\em stable} approximations where the inner-outer decomposition is preserved in the limit:

\begin{definition}
Suppose $\{ F_n \} \subset \mathscr J$ is a sequence of inner finctions which converge uniformly on compact subsets of the disk to an inner function $F$. We say that $F_n$ is a {\em (Nevanlinna) stable approximation} of $F$ if
\begin{equation}
\label{eq:goodness-condition}
\inn F' = \lim_{n \to \infty} \inn F_n', \qquad
\out F' = \lim_{n \to \infty} \out F_n'.
\end{equation}
\end{definition}

In general, we have inequalities in one direction:

\begin{theorem}
\label{approximation-thm}
Suppose $\{ F_n \} \subset \mathscr J$ is a sequence of finite Blaschke products which converge uniformly on compact subsets of the disk to a holomorphic function $F: \mathbb{D} \to \mathbb{D}$. Also assume that the  $B_n = \inn F'_n$ converge to an inner function $I$.
Then $F \in \mathscr J$ and the following inequalities hold:
\begin{align}
\label{eq:fa1}
\sigma(F') & \le \sigma(I), \\
\label{eq:fa2}
|\inn F'| & \ge |I|, \\
\label{eq:fa3}
\int_{\mathbb{S}^1} \log |F'| dm & \le \lim_{n \to \infty} \int_{\mathbb{S}^1} \log |F_n'| dm.
\end{align}
Furthermore, either all of the above inequalities are equalities or none of them are.

\end{theorem}

\begin{proof}
{\em Step 1.}
Taking  $n \to \infty$ in $\lambda_{F_n} \ge |B_n| \lambda_{\mathbb{D}}$ gives
$\lambda_{F} \ge |I| \lambda_{\mathbb{D}}$. Lemma \ref{technical-lemma} shows that $F$ is inner with $F' \in \mathcal N$.
For any $0<r<1$ and interval $E \subset \mathbb{S}^1$, we have
$$
\int_{rE} \log \frac{\lambda_{F}}{\lambda_{\mathbb{D}}} \, dm  \ge  \int_{rE} \log |I| dm.
$$
Taking $r \to 1$ and using Lemma \ref{nevanlinna-gap} as well as the {\em easy} part of Lemma \ref{derivative-gap} shows
$$
- \sigma(F')(E)  \, = \, - \int_{E} \log |F'| dm + \lim_{r \to 1} \int_{rE} \log |F'| dm  \, \ge \, -\sigma(I)(E),
$$
provided the endpoints of $E$ do not charge $\sigma(I)$ and $\sigma(F')$. This proves the first inequality (\ref{eq:fa1}).

\medskip

{\em Step 2.} Clearly, $\inn F'$ and $I = \lim (\inn F'_n)$ have the same zeros in the unit disk but may have different singular factors. However, it is easy to see that for singular inner functions, one has the inequality $S_1 \le S_2$ if and only if $\sigma(S_1) \ge \sigma(S_2)$. The ``if'' direction is obvious, while the ``only if'' direction follows from the identity $$0 \, \le \, \lim_{r \to 1} \int_{rE} \log|S_2/S_1| dm \, = \, -\sigma(S_2)(E) + \sigma(S_1)(E),$$ valid for any generic interval $E \subset \mathbb{S}^1$ whose endpoints do not charge the measures $\sigma(S_1)$ and $\sigma(S_2)$.
This proves (\ref{eq:fa2}) and shows that the equality cases in (\ref{eq:fa1}) and (\ref{eq:fa2}) coincide.

Since $F'_n \to F'$ uniformly on compact subsets of the disk, (\ref{eq:fa2})  is equivalent to the inequality $|\out F'(z)| \le \bigl |\lim_{n \to \infty} \out F_n'(z) \bigr|$. Setting $z = 0$ and taking logarithms gives (\ref{eq:fa3}). However, if (\ref{eq:fa3}) is an equality, then by the maximum modulus principle, we must have $|\out F'(z)| = \bigl |\lim_{n \to \infty} \out F_n'(z) \bigr|$ for all $z \in \mathbb{D}$, since outer factors do not vanish.
This completes the proof.
\end{proof}

 \begin{remark}
After we prove the fundamental lemma (Lemma \ref{fundamental-lemma}), the assumption that the $F_n$ be finite Blaschke products in
Theorem \ref{approximation-thm} will no longer be necessary.
 \end{remark}
 
For some applications, we need to slightly vary the assumptions in the above theorem:

\begin{theorem}
\label{approximation-thm2}
In the context of Theorem \ref{approximation-thm}, suppose instead that the $B_n$ converge to a non-zero holomorphic function $H: \mathbb{D} \to \mathbb{D}$ with the inner-outer decomposition $H = I \cdot O$. Then, the  inequalities (\ref{eq:fa1})--(\ref{eq:fa3}) still hold. The equality case in (\ref{eq:fa3}) implies that $\{F_n\}$ is a stable sequence, in particular, the outer factor $O = 1$ is trivial and the $B_n$ converge to an inner function.
\end{theorem}

The proof of \ref{approximation-thm2} is nearly identical to that of Theorem \ref{approximation-thm}, so we only sketch the details. Observe that since $\| H \|_\infty \le 1$, we have 
$|O(z)| \le 1$ and $|I(z)| \ge |H(z)|$ for $z \in \mathbb{D}$. Following Step 1 of the proof of Theorem \ref{approximation-thm}, we obtain the inequality
$\lambda_F \ge |I \cdot O| \lambda_{\mathbb{D}}$. We may still use Lemma \ref{technical-lemma} to conclude that
$F \in \mathscr J$. The remark after Lemma \ref{ac-theory1} allows us to conclude (\ref{eq:fa1}) and (\ref{eq:fa2}) in this more general case as well.
We may weaken (\ref{eq:fa2}) to $|\inn F'| \ge |H|$, which is equivalent to (\ref{eq:fa3}). This time, the equality case in (\ref{eq:fa3}) forces $I = H$ and $O = 1$.
 
 \begin{remark}
In view of Lemma \ref{technical-lemma}, if a sequence of finite Blaschke products $F_n$ converges to a function $F \not\in \mathcal N$, then $H = \lim(\inn F'_n)$ must be 0.
 \end{remark}
 
Craizer's argument from \cite[Lemma 5.4]{craizer} shows:

\begin{lemma}
\label{craizer-lemma}
Any inner function $F \in \mathscr J$ admits a stable approximation.
\end{lemma}

\begin{proof}
Suppose $\xi \in \mathbb{D}$ is such that $T_\xi \circ F$ is a Blaschke product, where $T_\xi = \frac{z - \xi}{1 - \overline{\xi}z}$.
We may choose a sequence $F_{n,\xi}$ of finite Blaschke products converging to $F$ so
that $T_\xi \circ F_{n, \xi}$ is a sequence of partial products of $T_\xi \circ F$. Then, for any $x \in \mathbb{S}^1$, we have
$$
|(T_\xi \circ F_{n,\xi})'(x)| \le |(T_\xi \circ F)'(x)|,
$$
see \cite[Corollary 4.13]{mashreghi}. It follows that
$$
|(F_{n,\xi})'(x)|  \le \biggl [ \frac{1 + |\xi|}{1-|\xi|} \biggr ]^2 |F'(x)|,
$$
which leads to the estimate
$$
\int_{\mathbb{S}^1} \log |F_{n,\xi}'(x)| dm  \le 2 \log \frac{1+|\xi|}{1-|\xi|} + \int_{\mathbb{S}^1} \log |F'(x)| dm.
$$
Since we can choose $\xi$ arbitrarily close to 0 (e.g.~see \cite[Theorem 2.5]{mashreghi}), we can diagonalize to find a sequence $F_n$ converging to $F$ for which
$$
\limsup_{n \to \infty} \int_{\mathbb{S}^1} \log |F_{n}'(x)| dm  \le  \int_{\mathbb{S}^1} \log |F'(x)| dm.
$$
However, by Theorem \ref{approximation-thm2}, the lower bound is automatic and the sequence $\{F_n\}$ is stable.
\end{proof}

\begin{remark}
Since translation $f \to T_\xi \circ f$ is continuous in $L^1(\mathbb{S}^1)$, the above proof shows that $\log|F_n'| \to \log|F'|$ converges in the $L^1$-norm.
\end{remark}

With the help of a Nevanlinna stable approximation $F_n \to F$, we can deduce
(\ref{eq:fundamental-lemma}) by taking $n \to \infty$ in
$
\lambda_{F_n} \ge |\inn F_n'| \lambda_{\mathbb{D}}.
$
Since minimality was proved in Section \ref{sec:inject}, the proof of the fundamental lemma (Lemma \ref{fundamental-lemma}) is complete.

 We can endow the space of analytic functions $\mathscr E = \{ f : f' \in \mathcal N\}$ with the {\em strong stable topology} by specifying that $f_n \to f$ if the $f_n$ converge uniformly on compact sets to $f$ and $\log|f_n'| \to \log|f'|$ in the $L^1(\mathbb{S}^1)$-norm.
By the above remark, finite Blaschke products are dense in $\mathscr J$,
while Theorem \ref{approximation-thm} implies that the subset $\mathscr J \subset \mathscr E$ is closed (see the remark after the theorem).
Another possible topology on $\mathscr E$ is the {\em weak stable topology} where one only requires the weak-$*$ convergence of measures $\log|f_n'|dm \to \log|f'|dm$. The above properties also hold in this topology.

\subsection{Example of an unstable approximation}

We now give an example of a sequence of finite Blaschke products which is not Nevanlinna stable.
Let $F_n$ be the Blaschke product of degree $n+1$ with zeros at the origin and at $z_j = e^{j  (2 \pi  i/n)} \cdot (1-1/n^2)$, $j=1,2,\dots, n$. With the normalization $F_n'(0) > 0$, the maps $F_n$ converge to the identity since $\sum_{j=1}^n (1-|z_j|) \to 0$ as $n \to \infty$. Recall that for $x \in \mathbb{S}^1$, one has the formula $|F_n'(x)| = 1 + \sum_{j=1}^n P_{z_j}(x)$, where $P_z$ is the Poisson kernel as viewed from $z \in \mathbb{D}$, e.g.~see \cite[Theorem 4.15]{mashreghi}. Computations show $$\int_{I_j} \log|F_n'| dm \, \ge \, \int_{I_j} \log|1 + P_{z_j}| dm \, \gtrsim \, 1/n$$ where $I_j$ consists of the points on the unit circle for which the closest zero is $z_j$. Hence, $|\out F_n'(0)| = \exp \int_{\mathbb{S}^1} \log |F'_n|dm > c > 1$ for some constant $c$ independent of $n \ge 1$.
Since the outer parts $\out F_n'$ do not converge to the constant function 1, neither can the inner parts $\inn F_n'$.

 A modification of this construction can be used to show the existence of a sequence of finite Blaschke products $F_n \to z$ (and thus $F_n' \to 1$) for which $\inn F_n' \to S_{\delta_1}$ and $\out F_n' \to 1/S_{\delta_1}$.

\section{Understanding the image}
\label{sec:understanding-image}

In this section, we discuss the image of the map $F \to \inn F'$ and prove the decomposition, product and division rules from the introduction. We also show that the map $F \to \inn F'$ is not surjective by exhibiting a large class of invisible measures. A complete description of the image will be given in the next section.

\subsection[Wedging F\_mu with F\_C]{Wedging $F_\mu$ with $F_C$}

\begin{theorem}
\label{wedge-fufc}
{\em (i)} Suppose $F_\mu \in \mathscr J$ is an inner function with $\inn F_\mu' = S_\mu$. Let $F_{\mu, C} = F_\mu \wedge F_C$ where $C$ is a Blaschke sequence. Then, $ \inn F_{\mu, C}' = B_C S_\mu$.

{\em (ii)} Conversely, if $F_{\mu,C} \in \mathscr J$ is an inner function with $\inn F_{\mu,C}' = B_C S_\mu$, then there exists an inner function $F_\mu$ with $ \inn F_{\mu}' = S_\mu$.
\end{theorem}

\begin{proof}
(i) Since $\lambda_{F_C} \ge \lambda_{F_{\mu, C}} \ge |B_C|\lambda_{F_\mu}$, the critical set of $F_{\mu, C}$ is precisely $C$ with the correct multiplicity; while the inequalities
$
\lambda_{F_\mu}  \ge  \lambda_{F_{\mu,C}}  \ge  |B_C| \lambda_{F_\mu}
$
show that $\sigma(F'_{\mu, C})  = \mu$. Hence  $ \inn F_{\mu, C}' = B_C S_\mu$ as desired.

(ii) Suppose $F_{\mu, C} \in \mathscr J$ is an inner function with $\inn F'_{\mu, C} = B_C S_\mu$. 
Let $F_{n}$ be some approximation of $F_{\mu, C}$ by finite Blaschke products (stability is not required in this proof).
 For any $0 < r < 1$, we can  form the sequence of finite Blaschke products
 $F_{n,r}$ by removing the critical points from $F_n$ that lie in the ball $\{ z: |z|<r\}$, and considering the maximal Blaschke product with the remaining critical points (with the normalization $F_{n,r}(0) = 0$ and $F_{n,r}'(0) > 0$).
For each $r$, we pick a subsequential limit $F_r$ of $F_{n,r}$. We may then extract a further subsequential limit $F$ by taking $r \to 1$.
By construction, we have
$$
|B_C| \lambda_{F} \le \lambda_{F_{\mu,C}} \le \lambda_F.
$$
Since the limit $F$ cannot be constant, by Hurwitz' theorem, $F$ has no critical points.
 The above inequalities imply $\sigma(F') = S_\mu$ and therefore $\inn F' = S_\mu$.
\end{proof}

\subsection{Subseqences of stable sequences}

In the next lemma, we show that any subsequence of a stable sequence is also stable:

\begin{lemma}
Suppose that $F_{C_n} \to F_{\mu_1+\mu_2}$ is a stable sequence. Suppose that $C_{1,n} \subset C_n$ is such that $B_{C_{1,n}}$ converges to $S_{\mu_1}$. Then, $F_{C_{1,n}} \to F_{\mu_1}$.
\end{lemma}

\begin{proof}
Write $C_n = C_{1,n} \cup C_{2,n}$. From the assumptions, $B_{C_{1,n}} \to S_{\mu_1}$ and $B_{C_{2,n}} \to S_{\mu_2}$. After passing to a subsequence, we can ensure convergence: $$F_{C_{1,n}} \to F_{\nu_1}, \qquad \nu_1 \le \mu_1,$$
$$F_{C_{2,n}} \to F_{\nu_2}, \qquad \nu_2 \le \mu_2.$$
The monotonicity of limits follows from Theorem \ref{approximation-thm}.

For each $n$, we have $\lambda_{F_{C_n}} \ge |B_{1,n}| \lambda_{F_{C_{2,n}}}$ and therefore, after taking $n \to \infty$, we see that  $$\lambda_{F_{\mu_1+\mu_2}} \ge |S_{\mu_1}| \lambda_{F_{\nu_2}}.$$ As is now standard, we may deduce
$$
\mu_1+\mu_2 \le \mu_1 + \nu_2
$$
by examining the equation
$$
0 \le \lim_{r \to 1} \int_{rI} \log \frac{\lambda_{F_{\mu_1+\mu_2}}}{|S_{\mu_1}|\lambda_{ F_{\nu_2}}} \, dm.
$$
Hence $\nu_2 = \mu_2$ (and similarly $\nu_1 = \mu_1$) as desired.
\end{proof}

The above lemma has a number of consequences:

\begin{corollary}
\label{divisors}
If a measure $\mu$ is constructible, i.e.~if $F_\mu$ exists, then all $\nu \le \mu$ are also constructible. Combining with  Theorem \ref{wedge-fufc}, we see that the image of the mapping $F \to \inn F'$ is closed under taking divisors.
\end{corollary}

Indeed, given a stable approximation $F_{C_n}$ to $F_\mu$, it is not difficult to select $C_{1,n} \subset C_n$ so that
$B_{C_{1,n}} \to S_\nu$.

\begin{corollary}
\label{product-law}
If $F_{\mu_1}$ and $F_{\mu_2}$ are constructible, then $F_{\mu_1 + \mu_2}$ is also constructible.
\end{corollary}

The proof relies on the Solynin-type estimate
\begin{equation}
\label{eq:sol}
\lambda_{F_{C_1}}\lambda_{F_{C_2}} \ge \lambda_{F_{C_1 \cup C_2}} \lambda_{F_{C_1 \cap C_2}},
\end{equation}
valid when $C_1$ and $C_2$ are finite subsets of the disk counted with multiplicity. The proof of (\ref{eq:sol}) is essentially that of \cite[Lemma 2.8]{KR-solynin}, so we only sketch the details. Consider the function
$$
u(z) = \log^+ \biggl ( \frac{\lambda_{F_{C_1 \cup C_2}} \lambda_{F_{C_1 \cap C_2}}}{\lambda_{F_{C_1}}\lambda_{F_{C_2}}} \biggr ), \qquad z \in \mathbb{D}.
$$
We claim that it is subharmonic and non-negative in $\mathbb{D}$ yet tends to 0 as $|z| \to 1$. This will show that it is equal to 0 identically. It is clearly non-negative by definition. To show that $u(z)$ is subharmonic, one can check that $\Delta u \ge 0$. We refer the reader to \cite[Lemma 2.8]{KR-solynin} for the computation. For the last statement, note that by Lemma \ref{fundamental-lemma}, for a finite Blaschke product, the quotient $\lambda_{F}/\lambda_{\mathbb{D}} \to 1$ uniformly as $|z| \to 1$.

\begin{proof}[Proof of Corollary \ref{product-law}]
Choose approximations $F_{C_{1,n}} \to F_{\mu_1}$ and $F_{C_{2, n}} \to F_{\mu_2}$ by finite Blaschke products. Making a small perturbation if necessary, we can assume that the sets $C_{1,n}$ and $C_{2,n}$ are disjoint. Let $C_n = C_{1,n} \cup C_{2,n}$ be their union. Passing to a subsequence, we may assume that $F_{C_{n}} \to F_\mu$ for some measure $\mu$ on the unit circle. By Solynin's estimate, we have
\begin{equation}
\label{eq:product-law1}
\log \frac{\lambda_{\mathbb{D}}}{\lambda_{F_{C_{1,n}}}} + \log \frac{\lambda_{\mathbb{D}}}{\lambda_{F_{C_{2, n}}}} \le 
\log \frac{\lambda_{\mathbb{D}}}{\lambda_{F_{C_n}}}.
\end{equation}
Taking $n \to \infty$ gives
\begin{equation}
\label{eq:product-law2}
\log \frac{\lambda_{\mathbb{D}}}{\lambda_{F_{\mu_1}}} + \log \frac{\lambda_{\mathbb{D}}}{\lambda_{F_{\mu_2}}} \le 
\log \frac{\lambda_{\mathbb{D}}}{\lambda_{F_{\mu}}}.
\end{equation}
By examining averages over $rI$ and taking $r \to 1$, we discover that $\mu \ge \mu_1+\mu_2$. Applying Corollary \ref{divisors} shows that the measure $\mu_1+\mu_2$ is constructible.
\end{proof}

\begin{corollary}
If $S'_\mu \in \mathcal N$ then $\mu$ is constructible.
\end{corollary}

In \cite{cullen}, M.~Cullen showed that this is the case when the support of $\mu$ is a Beurling-Carleson set, that is, a closed subset of the unit circle of zero Lebesgue measure whose complement is a union of arcs $\bigcup_k I_k$ with $\sum |I_k| \log \frac{1}{|I_k|} < \infty$.

\subsection{Invisible measures}

Let $\mu$ be a finite positive measure on the unit circle, which is singular with respect to the Lebesgue measure. We say $\mu$ is {\em invisible} if for any measure
 $0 <  \nu \le \mu$, there does not exist a function $F_\nu \in \mathscr J$ with $\inn F_\nu' = S_\nu$.

\begin{lemma}
Either the map $F \to \inn F'$ is surjective or there exists an invisible measure.
\end{lemma}

\begin{proof}
Suppose $F_\mu$ is not constructible. Since the hull of the metric $|S_\mu| \cdot \lambda_{\mathbb{D}}$ defined in Section \ref{sec:hull} cannot vanish anywhere, it must be of the form $\lambda_{F_{\nu}}$ for some measure $\nu$. (Lemma \ref{technical-lemma} explains why $F_\nu$ must be an inner function.) Applying Lemma \ref{ac-theory1}, we see that $\nu < \mu$ since equality cannot hold.
From the product rule (Corollary \ref{product-law}), it follows that the measure $\mu - \nu$ is invisible. More precisely, if $\sigma \le \mu - \nu$ was constructible, then $\lambda_{F_v} > \lambda_{F_{\mu-\sigma/2}} > |S_\mu| \cdot \lambda_\mathbb{D}$ would contradict the definiton of $\nu$.
\end{proof}

Actually, the above argument shows a little more:

\begin{theorem} 
\label{inv-criterion}
A measure $\mu$ is invisible if and only if the hull of $|S_\mu| \cdot \lambda_{\mathbb{D}}$   is the Poincar\'e metric. More generally, any measure $\mu$ can be uniquely decomposed into a constructible part and an invisible part: $\mu = \mu_{\vis} + \mu_{\inv}$, in which case, the hull of $|S_\mu| \cdot \lambda_{\mathbb{D}}$ is $\lambda_{F_{\mu_{\vis}}}$.
\end{theorem}

We are now in a position to prove the countable version of the product rule (Lemma \ref{product-law2}).
Suppose we are given countably many constructible measures $\mu_j$, $j=1, 2, \dots$ such that their their sum $\mu = \sum_{j=1}^\infty \mu_j$ is a finite measure.
 We claim that $\mu$ is constructible. According to Theorem \ref{inv-criterion}, the hull of $|S_{\mu}| \cdot \lambda_{\mathbb{D}}$ is of the form $\lambda_{F_{\nu}}$ for some measure $\nu \le \mu$.
However, from Corollary \ref{product-law}, we know that $\tilde \mu_j = \mu_1 + \mu_2 + \dots + \mu_j$ is constructible. This shows that $\nu \ge \tilde \mu_j$ for any $j$, which forces $\nu = \mu$.

\subsection{A criterion for invisibility}
\label{sec:cfi}

In this section, we only consider conformal metrics with strictly positive densities, that is, genuine metrics instead of pseudometrics.
Given a positive continuous function $u$ on $\mathbb{S}_r = \{z : |z|=r\}$, $0 < r < 1$, let $\Lambda_r[u]$ denote the unique conformal metric of curvature $-4$
on $\mathbb{D}_r = \{z : |z| < r\}$ which agrees with $u$ on $\mathbb{S}_r$. For the existence and uniqueness of $\Lambda_r[u]$, we refer the reader to \cite[Section 12]{heins} or \cite[Appendix]{conf-metrics}. For a non-vanishing SK-metric $\lambda$, we will sometimes write $\Lambda[\lambda] = \hat \lambda$ for the minimal metric of curvature $-4$ that exceeds $\lambda$.

\begin{lemma}
\label{mono-lemma}
The operation $u \to \Lambda_r[u]$ is monotone in $u$, that is, if $u \ge v$ then $\Lambda_r[u] \ge \Lambda_r[v]$.
\end{lemma}
To see this, note that the function
$
h = \log^+(\Lambda_r[v]/\Lambda_r[u])
$
is non-negative, subharmonic and identically zero on $\mathbb{S}_r$. As usual, to check that $h$ is subharmonic, we use the definition of curvature: 
$$\Delta h \, = \, (4 \Lambda_r[v]^2 - 4 \Lambda_r[u]^2) \cdot \chi_{v > u} \, \ge \, 0.$$
A similar argument shows:

\begin{lemma}
\label{lambda-convergence}
Let $\lambda$ be a non-vanishing conformal metric on the unit disk of curvature at most $-4$. For $0 < r < 1$, the metric  $\Lambda_r[\lambda(re^{i\theta})]$ is the minimal metric of curvature $-4$ that exceeds $\lambda$ on $\mathbb{D}_r$.
The family of metrics $\Lambda_r[\lambda(re^{i\theta})]$ is non-decreasing in $r$, and the limit
\begin{equation}
\label{eq:hull-formula}
\hat \lambda \, = \, \Lambda[\lambda] \, = \, \lim_{r \to 1} \Lambda_r[\lambda(re^{i\theta})]
\end{equation}
is the minimal metric of curvature $-4$ that exceeds $\lambda$ on $\mathbb{D}$.
\end{lemma}

In general, it is difficult to evaluate $\Lambda_r[u]$ explicitly. In the next lemma, we do so when $u$ is a constant function.

\begin{lemma}
\label{fail-lemma}
Given any $0 < c \le 1$, there exists a unique $0 < r' \le r$ so that
  $\Lambda_r [c \cdot \lambda_{\mathbb{D}}] = L^*\lambda_{\mathbb{D}}$ where
 $L(z) = \frac{r'}{r} \cdot z$ is the linear map $\mathbb{D}_{r} \to \mathbb{D}_{r'}$.
 \end{lemma}
  
 The lemma follows by observing that the metrics $(L_{r'})^*\lambda_{\mathbb{D}}$ are increasing in $r'$,
 so there is a unique value of $r'$ to make the boundary values agree. 
 
\begin{corollary}
\label{fail-lemma2}
We have
 $$\lim_{C \to \infty} \biggl [ \lim_{r \to 1} \frac{\Lambda_r[C]}{\lambda_{\mathbb{D}}} \biggr ] \to 1,$$
uniformly on compact subsets of the unit disk.
 \end{corollary}

With these preparations, we can now prove:

\begin{theorem}
\label{inv-thm}
Suppose $\mu$ is a singular measure on the unit circle which satisfies $\mu(I) \le C |I| \log|1/I|$ for any interval $I \subset \mathbb{S}^1$ and some constant $C > 0$.  Then, $\mu$ is invisible.
\end{theorem}

\begin{proof}
From the product rule (Corollary \ref{product-law}), it is easy to see that a measure $\mu$ is invisible if and only if $\varepsilon \cdot \mu$ is for any $\varepsilon > 0$. This allows us to assume that $\mu(I) \le \varepsilon |I| \log|1/I|$ which implies that the Poisson extension  
$P_\mu(z) \le \varepsilon(A \log \frac{1}{1-|z|} + B)$ for some constants $A$ and $B$. Hence, $|S_\mu| \lambda_{\mathbb{D}} \to \infty$ as $|z| \to 1$. The theorem now follows from the monotonicity principle (Lemma \ref{mono-lemma}) and Corollary \ref{fail-lemma2}.
\end{proof}

\section{Roberts decompositions}
\label{sec:roberts}

In this section, we show that if  $\mu$ does not charge Beurling-Carleson sets, then it is invisible, that is, any measure $0 < \nu \le \mu$ cannot be in the image of the map $F \to \inn F'$. To upgrade the argument of Section \ref{sec:cfi}, we will use the following theorem which is implicit in the work of Roberts \cite{roberts}:
\begin{theorem}
Suppose $\mu$ is a measure on the unit circle which does not charge Beurling-Carleson sets. Given a real number $c > 0$ and integer $j_0 \ge 1$, $\mu$ can be expressed as a countable sum
\begin{equation}
\label{eq:roberts-decomposition}
\mu = \sum_{j=1}^\infty \mu_j,
\end{equation}
where each $\mu_j$ enjoys an estimate on the modulus of continuity:
\begin{equation}
\label{eq:mod-cont2}
\omega_{\mu_j}(1/n_j) \le \frac{c}{n_j} \cdot \log n_j, \qquad n_j := 2^{2^{j+j_0}}.
\end{equation}
Here, $\omega_\mu(t) = \sup_{I \subset \mathbb{S}^1} \mu(I)$, with the supremum  taken over all intervals of length $t$. 
\end{theorem}

It will be important for us that the measure $\mu$ admits infinitely many decompositions with different parameters $c$ and $j_0$, where $c$ can be made arbitrarily small and $j_0$ arbitrarily large.

\begin{proof}[Sketch of proof.] For each $j = 1, 2, \dots$, we can define a partition $P_j$ of the unit circle into $n_j$ equal arcs. Since $n_j$ divides $n_{j+1}$, each next partition can be chosen to be a refinement of the previous one.
Given any measure $\mu$ on the unit circle, Roberts defines the notion of the {\em grating} of $\mu$ with respect to the sequence of partitions $(P_j)$. This procedure decomposes $\mu =  \sum_{j=1}^\infty \mu_j + \nu$ so that (\ref{eq:mod-cont2}) holds for each $j$, with the residual measure $\nu$ supported on the union of a Beurling-Carleson set  and a countable set.

To define $\mu_1$, consider all intervals in the partition $P_1$. Define an interval to be {\em light} if 
$\mu(I) \le (c/n_1) \cdot \log n_1$ and {\em heavy} otherwise. On a light interval, take $\mu_1 = \mu$, while
on a heavy interval, let $\mu_1$ be a  multiple of $\mu$ so that the mass $\mu_1(I) = (c/n_1) \cdot \log n_1$. Clearly, $\mu_1 < \mu$. 
Consider the difference  $\mu - \mu_1$ and grate it with respect to partition $P_2$ to form the measure $\mu_2$, then consider  $\mu - \mu_1 - \mu_2$ and grate it with respect to $P_3$ to form $\mu_3$, and so on. Continuing in this way, we obtain a sequence of measures $\mu_1, \mu_2, \dots$ where each next measure is supported on the heavy intervals of the previous generation.

By construction, the bound (\ref{eq:mod-cont2}) holds for all $j$. Inspection reveals that the residual measure $\nu$ is supported on the set of points which lie in heavy intervals at every stage. Up to a countable set, this coincides with $\mathbb{S}^1 \setminus \mathscr L$, where $\mathscr L$ is the union of the light intervals of any generation. (This  countable set consists of points on the unit circle which are endpoints of two different light intervals.) In \cite[Proof of Theorem 2]{roberts}, Roberts gave a simple computation using the relation $\log n_{j+1} = 2 \log n_j$ to show that $\mathbb{S}^1 \setminus \mathscr L$ is a Beurling-Carleson set.

Now, if $\mu$ does not charge Beurling-Carleson sets, it does not charge points so it cannot charge countable sets, which forces the residual measure to be 0.
\end{proof}

The estimate (\ref{eq:mod-cont2}) on the modulus of continuity is easily seen to be equivalent to an estimate on the Poisson extension: 
\begin{equation}
\label{eq:mod-cont}
|P_{\mu_j}| \le c' \cdot \log \frac{1}{1-|z|^2}, \qquad z \in B(0,1-1/n_j).
\end{equation}
Here, the constant $c'$ can be taken to be $cc_1$ for some $c_1 > 0$. This is stated in \cite[Lemma 2.2]{roberts}.

We will also need a simple lemma on conformal metrics:

\begin{lemma} {\em (i)} For any two singular measures $\mu_1$ and $\mu_2$ on the unit circle,
$$
\Lambda \Bigl [ | S_{\mu_1} | \cdot \Lambda \bigl [|S_{\mu_2}| \lambda_{\mathbb{D}} \bigr ] \Bigr] = \Lambda \bigl [|S_{\mu_1}| |S_{\mu_2}| \cdot \lambda_{\mathbb{D}} \bigr ].
$$

{\em (ii)} More generally, 
$$
\Lambda \Bigl [|S_{\mu_1}| \cdot \dots \Lambda \bigl [|S_{\mu_{j-1}}| \cdot \Lambda[|S_{\mu_j}| \lambda_{\mathbb{D}}] \bigr ] \dots \Bigr  ] = \Lambda \bigl [|S_{\mu_1}| |S_{\mu_2}|  \cdots |S_{\mu_j}| \cdot \lambda_{\mathbb{D}} \bigr ].
$$

{\em (iii)} For $\mu = \sum_{j=1}^\infty \mu_j$, we have
$$
\lim_{n \to \infty} \Lambda \Bigl [|S_{\mu_1}| \cdot \dots \Lambda \bigl [|S_{\mu_{j-1}}| \cdot \Lambda[|S_{\mu_j}| \lambda_{\mathbb{D}}] \bigr ] \dots \Bigr  ] 
 = \Lambda \bigl [|S_{\mu}| \lambda_{\mathbb{D}} \bigr ].
$$
\end{lemma}

\begin{proof}
(i) The $\ge$ direction follows from the monotonicity of $\Lambda$. For the $\le$ direction, it suffices to show that
$$
 | S_{\mu_1} | \cdot \Lambda \bigl [|S_{\mu_2}| \lambda_{\mathbb{D}} \bigr ]  \le \Lambda \bigl [|S_{\mu_1}| |S_{\mu_2}| \cdot \lambda_{\mathbb{D}} \bigr ]
$$
or
$$
 | S_{\mu_1} | \cdot \Lambda_r \bigl [|S_{\mu_2}| \lambda_{\mathbb{D}} \bigr ]  \le \Lambda_r \bigl [|S_{\mu_1}| |S_{\mu_2}| \cdot \lambda_{\mathbb{D}} \bigr ]
$$
for any $0 < r < 1$, cf.~Lemma \ref{lambda-convergence}. To this end, we form the function
$$
u_r = \log^+ \biggl ( \frac{ | S_{\mu_1} | \cdot \Lambda_r \bigl [|S_{\mu_2}| \lambda_{\mathbb{D}} \bigr ] }{ \Lambda_r \bigl [|S_{\mu_1}| |S_{\mu_2}| \cdot \lambda_{\mathbb{D}} \bigr ] } \biggr )
$$ 
defined on $\mathbb{D}_r = \{ z : |z| < r \}$. Since it is subharmonic and vanishes on $\mathbb{S}_r = \partial \mathbb{D}_r$, it must be identically 0. This proves the $\le$ direction.

(ii) follows after applying (i) $j-1$ times.

(iii) Let $\tilde \mu_j = \mu_1 + \mu_2 + \dots + \mu_j$. By part (i), we have
 $$
|S_{\mu-\tilde \mu_j}| \cdot \Lambda \bigl [|S_{\tilde \mu_j}| \lambda_{\mathbb{D}} \bigr ] \, \le \,
  \Lambda \bigl [|S_{\mu}| \lambda_{\mathbb{D}} \bigr ] \, \le \,  \Lambda \bigl [|S_{\tilde \mu_j}| \lambda_{\mathbb{D}} \bigr ].
 $$
 Since $|S_{\mu-\tilde \mu_j}| \to 1$, it follows that  $\Lambda \bigl [|S_{\tilde \mu_j}| \lambda_{\mathbb{D}} \bigr ]$
 are decreasing and converge to $\Lambda \bigl [|S_{\mu}| \lambda_{\mathbb{D}} \bigr ]$. The quantities on the left side also decrease to their limit. Therefore, the limits must coincide.
\end{proof}

With these preparations, we can now show Theorem \ref{not-charging-thm}:

\begin{proof}[Proof of Theorem \ref{not-charging-thm}.]
{\em Step 1.} Let $\mu = \mu_j$ be the Roberts decomposition (\ref{eq:roberts-decomposition}) with parameters $c$ and $j_0$ to be chosen later. 
In view of the invisibility criterion (Theorem \ref{inv-criterion}), it suffices to show that
\begin{equation}
 \label{eq:clever-strategy}
\lambda_j \, := \, \Lambda_{1-1/n_1} \biggl [|S_{\mu_1}| \cdot \dots \Lambda_{1-1/n_{j-1}} \Bigl [|S_{\mu_{j-1}}| \cdot \Lambda_{1-1/n_j} \bigl [|S_{\mu_j}| \cdot \lambda_{\mathbb{D}} \bigr ] \Bigr ] \dots \biggr ] 
\end{equation}
is close to the hyperbolic metric at the origin, uniform in $j \ge 1$. Indeed, by the monotonicity properties of $\Lambda$, we have 
\begin{equation}
\label{eq:clever-strategy2}
\lambda_j \, \le \, \Lambda \biggl [|S_{\mu_1}| \cdot \dots \Lambda \Bigl [|S_{\mu_{j-1}}| \cdot \Lambda \bigl [|S_{\mu_j}| \cdot \lambda_{\mathbb{D}} \bigr ] \Bigr ] \dots \biggr ],
\end{equation}
so that if $\lambda_j$ is close to $\lambda_{\mathbb{D}}$, then so must  $$\Lambda \bigl [|S_{\mu_1}| |S_{\mu_2}|  \cdots |S_{\mu_n}| \lambda_{\mathbb{D}} \bigr ].$$

{\em Step 2.}
The estimate on the modulus of continuity of $\mu_j$ implies that $|S_{\mu_j}| \lambda_{\mathbb{D}} \ge \lambda_{\mathbb{D}}^{4/5}$ on the circle $\mathbb{S}_{1-1/n_j}$. Here, we use the fact that we can choose 
$c' < 1/10$ in (\ref{eq:mod-cont}).
We claim that this implies that
\begin{equation}
\label{eq:inductive-goal}
\Lambda_{1 - 1/n_j} \bigl [|S_{\mu_j}| \lambda_{\mathbb{D}} \bigr ] \ge (1/2) \lambda_{\mathbb{D}}, \qquad \text{on } \mathbb{S}_{1-1/n_{j-1}}.
\end{equation} Assuming (\ref{eq:inductive-goal}), we have
$$
|S_{\mu_{j-1}}| \cdot \Lambda_{1 - 1/n_j} \bigl [|S_{\mu_j}| \lambda_{\mathbb{D}} \bigr ] \ge \lambda_{\mathbb{D}}^{4/5}, \qquad \text{on }\mathbb{S}_{1-1/n_{j-1}}.
$$
 We could then inductively show that $\lambda_j \ge (1/2)  \lambda_{\mathbb{D}}$ on $\mathbb{S}_{1-1/n_1}$. By Corollary \ref{fail-lemma2}, this would mean that $\lambda_j$ is very close to $\lambda_{\mathbb{D}}$ at the origin, provided $n_1$ is large (this is where we use that $j_0$ can be made arbitrarily large.)
 
 \medskip

{\em Step 3.}
Thus, we need to show that  $\Lambda_{1-1/n_j} \bigl [ |S_{\mu_j}| \lambda_{\mathbb{D}} \bigr ] \ge (1/2) \cdot \lambda_{\mathbb{D}}$ on $\mathbb{S}_{1-1/n_{j-1}}$.
Define $\varepsilon > 0$ by $1 - 1/n_j = 1 - \varepsilon$ so that $1 - 1/n_{j-1} = 1 - \varepsilon^{1/2}$. There exists a unique $0 < \ell < 1$ so that $\Lambda_{1-1/n_j} \bigl [\lambda_{\mathbb{D}}^{4/5} \bigr ] = L^*\lambda_{\mathbb{D}}$ where $L(z) = \ell   z$. 
Inspection shows that $1-\ell \asymp \varepsilon^{4/5}$. Therefore,
$$
\Lambda_{1-1/n_{j}} \bigl [ |S_{\mu_j}| \lambda_{\mathbb{D}} \bigr ] 
\, \ge \,
\Lambda_{1-1/n_{j}} \bigl [ \lambda_{\mathbb{D}}^{4/5} \bigr ]
 \, = \, 
 \frac{\ell}{1-|\ell z|^2} \ge (1/2) \cdot \lambda_{\mathbb{D}}, \quad \text{on }\mathbb{S}_{1-1/n_{j-1}}$$
as desired. 
 \end{proof}

\newpage

\bibliographystyle{amsplain}

\begin{thebibliography}{00}

\bibitem{ahern-clark} P.~R.~Ahern, D.~N.~Clark, {\em On inner functions with $H_p$-derivative}\/, Michigan Math. J. 21 (1974), no. 2, 115--127.

\bibitem{AAN} A.~B.~Aleksandrov, J.~M.~Anderson, A.~Nicolau, {\em Inner functions, Bloch spaces and symmetric measures}\/, Proc. London Math. Soc. 79 (1999), no. 2, 318--352.


\bibitem{cullen}M.~Cullen, {\em Derivatives of singular inner functions}\/, Michigan Math. J. 18 (1971), no. 3, 283-287.

\bibitem{craizer}M.~Craizer, {\em Entropy of inner functions}\/, Israel J. Math. 74 (1991), no. 2, 129--168.

\bibitem{duren}P.~Duren, {\em Theory of $H_p$ spaces}\/, Dover Publications, 2000.

\bibitem{dyakonov-coinvariant}K.~M.~Dyakonov, {\em Smooth functions and co-invariant subspaces of the shift operator}\/, Algebra i Analiz 4 (1992), no. 5, 117--147, in Russian. English translation: St. Petersburg Math. J. 4 (1993), no. 5, 933--959.


\bibitem{dyakonov-reverse}K.~M.~Dyakonov, {\em A Reverse Schwarz--Pick Inequality}\/, Computational Methods and Function Theory 13 (2013), no. 7--8, 449--457.

\bibitem{dyakonov-mobius}K.~M.~Dyakonov, {\em A characterization of M\"obius transformations}\/, C. R. Math. Acad. Sci. Paris 352 (2014), no. 2, 593--595.

\bibitem{dyakonov-inner}K.~M.~Dyakonov, {\em Inner functions and inner factors of their derivatives}\/,  Integr. Equ. Oper. Theory 82 (2015), no. 2, 151--155.

\bibitem{heins} M.~Heins, {\em On a class of conformal metrics}, Nagoya Math. J. 21 (1962), 1--60.

\bibitem{korenblum} B.~Korenblum, {\em Cyclic elements in some spaces of analytic functions}\/, Bull. Amer. Math. Soc. 5 (1981), 317--318.

\bibitem{kraus}D.~Kraus, {\em Critical sets of bounded analytic functions, zero sets of Bergman spaces and nonpositive curvature}\/, Proc. London Math. Soc. 106 (2013), no. 4, 931--956.

\bibitem{KR-critical}D.~Kraus, O.~Roth, {\em Critical points of inner functions, nonlinear partial differential equations, and an extension of Liouville's theorem}\/, J. London Math. Soc. 77 (2008), no. 1, 183--202.

\bibitem{KR-survey}D.~Kraus, O.~Roth, {\em Critical Points, the Gauss Curvature Equation and Blaschke Products}\/, In: Blaschke Products and Their Applications, Fields Institute Communications 65 (2012), 133--157.


\bibitem{conf-metrics}D.~Kraus, O.~Roth, {\em Conformal metrics}\/, Lecture Notes Ramanujan Math. Society, Lecture Notes Series 19 (2013), 41--83.

\bibitem{KR-solynin}D.~Kraus, O.~Roth, {\em Strong submultiplicativity of the Poincar\'e metric}\/, J. Analysis 24 (2016), no. 1, 39--50.

\bibitem{maximal-blaschke}D.~Kraus, O.~Roth, {\em Maximal Blaschke Products}\/, Adv. Math. 241 (2013), 58--78.


\bibitem{mashreghi}J.~Mashreghi, {\em Derivatives of Inner Functions}\/, Fields Institute Monographs, 2012.






\bibitem{roberts} J.~W.~Roberts, {\em Cyclic inner functions in the Bergman spaces and weak outer functions in $H^p$,
$0 < p < 1$}\/, Illinois J. Math. 29 (1985), 25--38.





\end{thebibliography}

\end{document}